\documentclass[reqno]{amsart}

\usepackage{amsmath,enumerate}

\usepackage[colorlinks=true,allcolors=black]{hyperref}
\usepackage[nameinlink]{cleveref} 
\usepackage{amsaddr} 
\usepackage{tikz}

\theoremstyle{plain}
\newtheorem{theorem}{Theorem}[section]
\newtheorem{lemma}[theorem]{Lemma}
\newtheorem{a.1}[theorem]{A.1}

\theoremstyle{definition}

\theoremstyle{remark}

\begin{document}

	\title[Characterization of finite groups of prime exponent]{A combinatorial characterization of finite groups of prime exponent}
	
	\author[Ramesh Prasad Panda]{Ramesh Prasad Panda}

	\address{School of Mathematical Sciences\\ National Institute of Science Education\\ and Research Bhubaneswar (HBNI)\\
		P.O. Jatni, District - Khurda\\ Odisha - 752050, India}

	\email{rppanda@niser.ac.in}

\begin{abstract}
	The power graph of a group $G$ is a simple and undirected graph with vertex set $G$ and two distinct vertices are adjacent if one is a power of the other. In this article, we characterize (non-cyclic) finite groups of prime exponent and finite elementary abelian $2$-groups (of rank at least $2$) in terms of their power graphs. 
\end{abstract}

	\maketitle
	
	\section{Introduction}

Let $G$ be a group. The \emph{power graph} $\mathcal{G}(G)$ of $G$ is a simple and undirected graph whose vertex set is $G$ and distinct vertices $u$ and $v$ are adjacent if $v=u^n$ for some $n \in \mathbb{N}$ or $u=v^m$ for some $m \in \mathbb{N}$. The notion of (directed) power graph of a group was introduced by Kelarev and Quinn \cite{kelarev2000combinatorial, kelarevDirectedSemigr}. Afterwards, Chakrabarty et al. \cite{GhoshSensemigroups} gave the above definition of power graph of a group. Recently, many interesting results on power graphs have been obtained, for instance, see \cite{curtin2014edge,feng2015,ma2018,moghaddamfar2013} and the references therein.

In the study of graphs constructed from groups, a useful topic is the extent to which structure of a group is reflected by the corresponding graph. By considering it for power graphs, Cameron \cite{Cameron} proved that if two finite groups $G_1$ and $G_2$ have isomorphic power graphs, they have same number of elements of each order (in particular, have the same spectra). Additionally, if $G_1$ and $G_2$ are abelian, it was shown by Cameron and Ghosh \cite{Ghosh} that $G_1$ and $G_2$ are isomorphic. Mirzargar et al. \cite{mirzargar2012power} proved that finite simple groups and finite cyclic groups are uniquely determined by their power graphs. In this article, we characterize non-cyclic finite groups of prime exponent (cf. Theorem \ref{min.pg.all2}) and finite elementary abelian $2$-groups of rank at least $2$ (cf. Theorem \ref{min.vercon}) in terms of connectedness of their power graphs.

The terminology used throughout are standard; for instance, we can refer to \cite{Dummit} for groups and \cite{west1996graph} for graphs.

For any graph $\Gamma$, its (vertex) connectivity, edge-connectivity and minimum degree are denoted by $\kappa(\Gamma)$, $\kappa'(\Gamma)$ and $\delta(\Gamma)$, respectively. Let $\Gamma$ be a non-trivial and connected graph. Then it is said to be \emph{minimally connected} if $\kappa(\Gamma-\varepsilon) = \kappa(\Gamma)-1$ for every edge $\varepsilon$ of $\Gamma$. Analogously, $\Gamma$ is said to be \emph{minimally edge-connected} if  $\kappa'(\Gamma-\varepsilon)= \kappa'(\Gamma)-1$ for every edge $\varepsilon$ of $\Gamma$. 
These graphs have been explored in various problems in extremal and structural graph theory (cf. \cite{bollobas2004graph, kriesell2013minimal}).
 We employ these notions on power graphs to study the corresponding finite groups. 

Now we state our main results. 	 
\begin{theorem}\label{min.pg.all2}
	A finite group $G$ is a non-cyclic group of prime exponent if and only if  $\mathcal{G}(G)$ is non-complete and minimally edge-connected.
\end{theorem}

\begin{theorem}\label{min.vercon}
	A finite group $G$ is an elementary abelian $2$-group of rank at least $2$ if and only if  $\mathcal{G}(G)$ is non-complete and minimally connected.
\end{theorem}

It is trivial to see that a finite group $G$ is a cyclic group of prime exponent $p$ if and only if  $\mathcal{G}(G)$ is a complete graph on $p$ vertices. A similar statement can be made about a finite elementary abelian $2$-group of rank $1$.	

\section{Preliminaries}
\label{prelim}

Let $G$ be a group with an element $x$. The order of $x$ in $G$ and the degree of $x$ in $\mathcal{G}(G)$ are denoted by $\mathrm{o}(x)$ and $\deg(x)$, respectively. We denote by $[x]$ the set of generators of the cyclic subgroup $\langle x \rangle$. Since an element of $G$ is a vertex of $\mathcal{G}(G)$ and vice versa, we use the terms element and vertex interchangeably. Let $\Gamma$ be a graph with vertex set $V(\Gamma)$ and edge set $E(\Gamma)$. For any subgraph $\Gamma'$ and a set of edges $S$ of $\Gamma$, we denote $\Gamma' - \{S \cap E(\Gamma')\}$ simply by $\Gamma' - S$.

We next recall some necessary results on power graphs of finite groups. 

\begin{lemma}[\cite{GhoshSensemigroups,power2017conn,power2017mindeg}]\label{lemma1}
	Suppose that $G$ is a finite group.
	\begin{enumerate}[\rm(i)]
		\item  
		$\mathcal{G}(G)$ is complete if and only if $G$ is a cyclic group of order one or prime power.
		\item 
		For any induced subgraph $\Gamma$ of $\mathcal{G}(G)$ with identity vertex, $\kappa'(\Gamma)=\delta(\Gamma)$.
		\item   
		Any minimum separating set of $\mathcal{G}(G)$ is of the form $\cup_{i=1}^s [x_i]$ for $x_1,\ldots, x_s$ in $G$.
	\end{enumerate}
\end{lemma}

We use Lemma \ref{lemma1}(i) without referring to it explicitly.

Given any group $G$, we denote by $\mathcal{G}^*(G)$, the subgraph of $\mathcal{G}(G)$ obtained by deleting the identity vertex. 
\begin{lemma}[{\cite[Theorem 4]{doostabadi2015power}}]\label{p.exponent}
	For any finite group $G$, $\mathcal{G}^*(G)$ is regular if and only if  $G$ is a cyclic group of prime power order or $G$ is of prime exponent. 
\end{lemma}

\begin{lemma}[{\cite[Proposition 3.1]{power2017conn}}]\label{lemma3}
	Suppose that $G$ is a finite $p$-group for some prime $p$. If $x$ is an element of order $p$ in $G$, then  $x$ is adjacent to all other vertices of the component of $\mathcal{G}^*(G)$ that contains $x$.
\end{lemma}

\section{Proofs of the main results}
\label{proof.main.results}

In this section $G$ is a non-trivial group with identity element $e$. We first prove Theorem \ref{min.pg.all2} and then Theorem \ref{min.vercon}. We begin with an essential lemma.

\begin{lemma}\label{min.necessary}
	Let $G$ be a finite group such that $\mathcal{G}(G)$ is minimally edge-connected.
	\begin{enumerate}[\rm(i)]
		\item If $x \in G$ and $\mathrm{o}(x) > 2$, then $\deg(x) = \delta(\mathcal{G}(G))$.
		\item If $\langle y \rangle$ is a maximal cyclic subgroup of $G$ and $\mathrm{o}(y) > 2$,  then $\mathrm{o}(y) = \delta(\mathcal{G}(G))+1$.
	\end{enumerate}
\end{lemma}

\begin{proof}
	(i) Note that $[x]$ has two or more elements. Let $\varepsilon$ be the edge in $\mathcal{G}(G)$ with endpoints $x$ and $x' \in [x]-\{x\}$. Then by Lemma \ref{lemma1}(ii) and the fact that $\mathcal{G}(G)$ is minimally edge-connected, we have $\delta(\mathcal{G}(G)-\varepsilon) = \delta(\mathcal{G}(G))-1$. Thus at least one endpoint of $\varepsilon$ has the minimum degree in $\mathcal{G}(G)$. Furthermore, as $\deg(x) = \deg(x')$, we get  $ \deg(x) = \delta(\mathcal{G}(G))$.

	\smallskip
	
	\noindent
	(ii) Observe that $\deg(y) = \mathrm{o}(y)-1$. Hence by (i), the proof follows.
\end{proof}

\begin{lemma}\label{no.maximal}
	Let $G$ be a finite group with no maximal subgroup of order two. If $\mathcal{G}(G)$ is minimally edge-connected, then $G$ is of prime power order.
\end{lemma} 

\begin{proof}
	It is known that every element of $G$ is in some of its maximal cyclic subgroups. Thus it follows from Lemma \ref{min.necessary}(ii) that the exponent of $G$ is $\delta(\mathcal{G}(G))+1$. The following arguments are similar to that of the proof of Lemma \ref{p.exponent}.
	
	We denote $d = \delta(\mathcal{G}(G))$. If possible, suppose $G$ is not of prime power order.
	Let $p_1 < p_2 < \ldots < p_r$, $r \geq 2$, be the prime factors of $d+1$. Then by Lemma \ref{min.necessary}(ii), $G$ has an element $x$ of order $\frac{d+1}{p_1}$. The degree of $x$ in $\mathcal{G}(G)$ is $\frac{d+1}{p_1}-1+m\phi(d+1)$, where $m$ is the number of maximal cyclic subgroups of $G$ containing $x$ and $\phi$ is the Euler's totient function. However, since $\mathrm{o}(x) > 2$, it follows from Lemma \ref{min.necessary}(i) that $\deg(x) = d$. As a result, $\frac{d+1}{p_1}+m\phi(d+1) =  d+1$, which yields $m(p_2-1)(p_3-1)\ldots(p_r-1) = p_2p_3\ldots p_r$. This implies that if $q$ is a positive divisor of $p_2-1$, then $q=p_i$ for some $2 \leq i \leq r$, which is a contradiction. Hence $G$ is a group of prime power order.
\end{proof}

Now we provide the proof of Theorem \ref{min.pg.all2}.

\begin{proof}[Proof of Theorem \ref{min.pg.all2}]
	
	Let $\mathcal{G}(G)$ be non-complete and minimally edge-connected. If $G$ has a maximal subgroup of order two, then $\kappa'(\mathcal{G}(G)) = 1$. Thus $\mathcal{G}(G)$ is a tree. Consequently, all elements of $G$ have order two, so that $G$ is an elementary abelian $2$-group of rank at least $2$.

	Now suppose $G$ has no maximal subgroup of order two. Then by Lemma \ref{no.maximal}, $G$ is a finite $p$-group for some prime $p$. If $p > 2$, then it follows from Lemma \ref{min.necessary}(i) that $\mathcal{G}^*(G)$ is regular. 
	
	We next take $p=2$. If $\mathcal{G}(G)$ has more than one block, then the only vertex common to any two distinct blocks is $e$.  Furthermore, by Lemma \ref{lemma3}, every block of $\mathcal{G}(G)$ has exactly one vertex that has order two in $G$. 
	
	Note   that for any $x$ in $G$, all elements of $\langle x \rangle$ are vertices in the same block of $\mathcal{G}(G)$.	If possible, suppose $\langle x_1 \rangle$ and $\langle x_2 \rangle$ are distinct maximal cyclic subgroups of $G$ whose elements are vertices in the same block, say $\Gamma$, of $\mathcal{G}(G)$. Let $y$ be the vertex in $\Gamma$ having order two in $G$ and $\varepsilon$ be the edge with endpoints $e$ and $y$. Then by assumption, $\langle y \rangle$ is not a maximal subgroup in $G$. Thus $\mathcal{G}(G)-\varepsilon$ and in particular, $\Gamma-\varepsilon$ is connected. We continue to write $d =\delta(\mathcal{G}(G))$. By applying Lemma \ref{lemma3} and Lemma \ref{min.necessary}(i), we have $d \geq 3$. Suppose $S$ is a minimum disconnecting set of $\mathcal{G}(G)-\varepsilon$. In view of  Lemma \ref{lemma1}(ii), $|S| = d-1$. If $\Gamma$ is the only block of $\mathcal{G}(G)$, then $\Gamma = \mathcal{G}(G)$. We observe that the minimum degree (hence the edge-connectivity) of any block of $\mathcal{G}(G)$ is $d$. Accordingly, if $\mathcal{G}(G)$ has any block $\Gamma'$ different from $\Gamma$, then $\Gamma'-S$ is connected. Thus we deduce that $(\Gamma-\varepsilon)-S$ is disconnected. Moreover, as  $\kappa'(\Gamma-\varepsilon) = d-1$, all elements of $S$ are therefore edges in $\Gamma$.
	
	Since all elements of $\langle x_1 \rangle$ and $\langle x_2 \rangle$ are vertices in $\Gamma$, by applying Lemma \ref{min.necessary}(ii), we get $|V(\Gamma)| \geq d+2$ (in fact, $|V(\Gamma)| \geq d+3$). So considering Lemma \ref{lemma3}, $e$ is connected to $y$ by a path of length two in $(\Gamma-\varepsilon)-S$. Because $(\Gamma-\varepsilon)-S$ is disconnected, there exists a vertex $z$ in $\Gamma$, different from $e$ and $y$, that is not connected to $e$ by any path in $(\Gamma-\varepsilon)-S$. As a result, $S$ contains the edge with endpoints $e$ and $z$ as well as the edge with endpoints $y$ and $z$. Since $\deg(z) = d$ and $|S|=d-1$, there exists a path of length two between $e$ and $z$ in $(\Gamma-\varepsilon)-S$. However, this contradicts the initial assumption about $z$. We thus conclude that every block of $\mathcal{G}(G)$ is a clique induced by a maximal cyclic subgroup of $G$. Consequently, in view of Lemma \ref{min.necessary}(ii), $\mathcal{G}^*(G)$ is regular. Additionally, $\mathcal{G}(G)$ is non-complete. Hence it follows from Lemma \ref{p.exponent} that $G$ is a non-cyclic group of prime exponent.
	
	Conversely, let $G$ be a non-cyclic group of prime exponent $p$. Then two non-identity vertices $u$ and $v$ are adjacent in $\mathcal{G}(G)$ if and only if $\langle u \rangle = \langle v \rangle$. Thus $\mathcal{G}(G)$ is union of finitely many maximal cliques of size $p$ with any two distinct maximal cliques having $e$ as the only common vertex. Therefore, $\mathcal{G}(G)$ is non-complete and minimally edge-connected.
\end{proof}

We next give a simple lemma and then prove Theorem \ref{min.vercon}.

\begin{lemma}\label{min.sepset.endpts}
	Let $\Gamma$ be a graph with an edge $\varepsilon$ such that $\Gamma-\varepsilon$ is connected. If $\kappa(\Gamma - \varepsilon) = \kappa(\Gamma)-1$, then no minimum separating set of $\Gamma - \varepsilon$ contains endpoints of $\varepsilon$.
\end{lemma}

\begin{proof}
	Suppose $u$ and $v$ are the endpoints of $\varepsilon$. If possible, let $S$ be a minimum separating set of $\Gamma - \varepsilon$ containing at least one of $u$ or $v$. We have $(\Gamma - \varepsilon)-S = \Gamma-S$, so that $S$ is a separating set of $\Gamma$. This implies $\kappa(\Gamma) \leq \kappa(\Gamma - \varepsilon)$, which contradicts the given condition. This proves the lemma.
\end{proof}

\begin{proof}[Proof of Theorem \ref{min.vercon}]
	Let $\mathcal{G}(G)$ be non-complete and minimally connected. We notice that $|G| \geq 4$. Let $x$ be a non-identity element of $G$ and if possible, let there exist $y \neq x$ such that $\langle x \rangle = \langle y \rangle$. Suppose $\varepsilon$ is the edge with endpoints $x$ and $y$. Since $x$ and $y$ are connected by the path $x,e,y$, $\mathcal{G}(G)-\varepsilon$ is connected.	By using the fact that $\mathcal{G}(G)$ is minimally connected and  $\kappa(\mathcal{G}(G)) \leq |G|-2$, we have $\kappa(\mathcal{G}(G)-\varepsilon) \leq |G|-3$. Let $S$ be a minimum separating set of $\mathcal{G}(G)-\varepsilon$. Then $\mathcal{G}(G)-S$ is a connected graph with three or more vertices, and by Lemma \ref{min.sepset.endpts}, $x,y \notin S$. We thus deduce that there is no path in $\mathcal{G}(G)-S$ that has vertices $x$ and $y$, and does not have the edge $\varepsilon$. Let $z$ be a vertex different from $x$ and $y$ in $\mathcal{G}(G)-S$. Then in $\mathcal{G}(G)-S$, $z$ is connected to exactly one of $x$ and $y$ by a path that does not have $\varepsilon$. Without loss of generality, let $z$ be connected to $x$ by a path that does not have $\varepsilon$ in $\mathcal{G}(G)-S$. Then $S \cup \{x\}$ is a separating set of $\mathcal{G}(G)$. Since $\kappa(\mathcal{G}(G)-\varepsilon)=\kappa(\mathcal{G}(G))-1$, we conclude that $S \cup \{x\}$ is a minimum separating set of $\mathcal{G}(G)$. Then by Lemma \ref{lemma1}(iii), $[x] \subseteq S \cup \{x\}$, which contradicts the fact that $y \notin S$. Accordingly, we get $[x]=\{x\}$. Because $|[x]|=\phi(\mathrm{o}(x))$ and $x \neq e$, we have $\mathrm{o}(x)=2$. Hence $G$ is an elementary abelian $2$-group. Additionally, as  $\mathcal{G}(G)$ is non-complete, the rank of $G$ is at least $2$. 
	
	For converse, suppose $G$ is an elementary abelian $2$-group of rank at least $2$. Then $\mathcal{G}(G)$ is a star on three or more  vertices. As a result, $\kappa(\mathcal{G}(G))=1$ and $\mathcal{G}(G)-\varepsilon$ is disconnected for every edge $\varepsilon$ of $\mathcal{G}(G)$. Therefore we conclude that $\mathcal{G}(G)$ is non-complete and minimally connected.
\end{proof}

\section*{Acknowledgement}
The author is thankful to Dr. K. V. Krishna for his constructive suggestions. This research work was partially supported by the fellowship of IIT Guwahati, India.

\end{document}